\documentclass[twoside,11pt]{article}
\usepackage{mathrsfs}
\usepackage{amssymb}
\usepackage{}%
\usepackage{latexsym}
\usepackage{amsmath,amssymb,epsfig,subfigure}
\usepackage{amsthm,enumerate,verbatim}
\usepackage{amsfonts}
\usepackage{color}
\usepackage{amsmath}
\usepackage{graphicx}
\usepackage{epstopdf}
\usepackage{float}
\usepackage{hyperref}
\usepackage{multirow}
\usepackage{threeparttable}
\usepackage{algorithm}
\usepackage{algorithmic}
\usepackage{graphicx}
\usepackage[numbers,sort&compress]{natbib}
\usepackage[figuresright]{rotating}
\usepackage[figuresright]{rotating}

\allowdisplaybreaks[4]
\providecommand{\U}[1]{\protect\rule{.1in}{.1in}}
\providecommand{\U}[1]{\protect\rule{.1in}{.1in}}
\newcommand{\BE}{\begin{equation}}
	\newcommand{\EE}{\end{equation}}

\numberwithin{equation}{section}
\newtheorem{proposition}{Proposition}[section]
\newtheorem{theorem}[proposition]{Theorem}
\newtheorem{lemma}[proposition]{Lemma}
\newtheorem{corollary}[proposition]{Corollary}
\newtheorem{definition}[proposition]{Definition}
\newtheorem{remark}[proposition]{Remark}

\newtheorem{example}[proposition]{Example}

\newenvironment{breakablealgorithm}
{
	\begin{center}
		\refstepcounter{algorithm}
		\hrule height0.8pt depth0pt \kern2pt
		\renewcommand{\caption}[2][\relax]{
			{\raggedright\textbf{Algorithm~\thealgorithm} ##2\par}%
			\ifx\relax##1\relax 
			\addcontentsline{loa}{algorithm}{\protect\numberline{\thealgorithm}##2}%
			\else 
			\addcontentsline{loa}{algorithm}{\protect\numberline{\thealgorithm}##1}%
			\fi
			\kern2pt\hrule\kern2pt
			
		}
	}{
		\kern2pt\hrule\relax
	\end{center}
}

\date{}
\begin{document}
	\title{
		Restarted Nonnegativity Preserving Tensor Splitting Methods via Relaxed Anderson Acceleration for Solving Multilinear Systems}
	\author{Dongdong Liu \thanks{%
			School of Mathematics and Statistics, Guangdong University of Technology, Guangzhou, China.
			({ddliu@gdut.edu.cn}).} \and Ting Hu$^{*}$\and Xifu Liu\thanks {Corresponding author, 	
			School of Mathematical Sciences, Chongqing Normal University, Chongqing, China. ({lxf211@cqnu.edu.cn})}}
	\maketitle
	\begin{abstract}
		Multilinear systems play an important role in scientific calculations of practical problems. In this paper, we consider a tensor splitting method with a relaxed Anderson acceleration for solving the multilinear systems. The new method preserves nonnegativity for every iterative step and improves the existing ones. Furthermore, the convergence analysis of the proposed method is given. The new algorithm performs effectively for numerical experiments.
	\end{abstract}
	
	\section{Introduction} \label{section1}
	Let $\mathcal {A}=(a_{i_1{i_2} \cdots {i_l}})$ be an order $l$ dimension $n$ tensor and ${\bf{b}} \in\mathbb{R}^n$. The multilinear systems were first given in \cite{wei2}:
	\begin{equation}\label{ax=b}
		\mathcal {A}{\bf{x}}^{l-1}={\bf{b}},
	\end{equation}
	where ${\bf{x}}$ is the unknown vector to be solved. The $i$-th entry of the vector $\mathcal {A}{\bf{x}}^{l-1}$ (see \cite{Qi3, QIBOOK}) is defined as follows:
	\begin{equation}\label{ax}
		{(\mathcal {A}{{\bf{x}}^{l - 1}})_i}  \equiv \sum\limits_{{i_2}, \cdots ,{i_l}=1}^n {{a_{i{i_2} \cdots {i_l}}}{x_{{i_2}}}}  \cdots {x_{{i_l}}},{\kern 1pt} {\kern 1pt} {\kern 1pt} {\kern 1pt} {\kern 1pt} {\kern 1pt} {\kern 1pt} i = 1,2,...,n,
	\end{equation}
	where $x_i$ denotes the $i$-th component of a vector ${\bf{x}}$. Since there are many applications such as tensor complementarity problems (e.g., see \cite{wei2,xie}), numerical partial differential equations (e.g., see \cite{wei2,lid2}) and Bellman equation (e.g., see \cite{bn, wang}), the nonlinear equation \eqref{ax=b} has attracted many researchers' attention, e.g., theoretical analysis of solutions  {(e.g., see \cite{wei2, lid1, liuaml})}, numerical algorithms such as splitting type methods (e.g., see \cite{llh, lid0}), {Gradient-based iterations (e.g., see \cite{wei2, xie, lid2, han, he, bai1,  liang, lv})}, dynamical approaches (e.g., see \cite{wang}) and so on.  The \textit{majorization matrix} of a {tensor $\mathcal {A}$ is denoted by $M(\mathcal {A})$ and
	$M(\mathcal {A})_{i_{1}i_{2}}=b_{i_{1}i_{2}...i_{2}}$ with $i_{1},i_{2}=1,2,...,n$.} For a vector ${\bf{y}}=(y_i)$, we define ${\bf{y}}^{[\frac{1}{l-1}]}=(y_1^{\frac{1}{l-1}},y_2^{\frac{1}{l-1}},...,y_n^{\frac{1}{l-1}})^{T}$. In \cite{llh}, several tensor splitting methods (TSM) were given via setting $\mathcal{A}=\mathcal{E}-\mathcal{F}$:
	\begin{equation}\label{sp0}
		{{\bf{x}}_{k}} = {(M{(\mathcal {E})^{ - 1}}\mathcal {F}{\bf{x}}_{k - 1}^{l - 1} + M{(\mathcal {E})^{ - 1}}{\bf{b}})^{[\frac{1}{{l - 1}}]}},
	\end{equation}
 The iterative solution of which is a fixed point of the following iterative function
	\begin{equation}\label{sp1}
		g_{\mathcal{E}}({{\bf{x}}}) = {(M{(\mathcal {E})^{ - 1}}\mathcal {F}{\bf{x}}^{l - 1} + M{(\mathcal {E})^{ - 1}}{\bf{b}})^{[\frac{1}{{l - 1}}]}},
	\end{equation}
	where $M{(\mathcal {E})^{ - 1}}\mathcal {F}$ is {\it{the iterative tensor}} of the splitting method \eqref{sp0}. {In order to accelerate the convergence of tensor splitting methods, a preconditioned splitting method was first proposed in \cite{com}, and then the preconditioned methods were further studied (e.g., see \cite{cui,cz,liuc}).} It is noted that tensor splitting type methods have a fine numerical performance in terms of the running time because of avoiding calculation of the Jacobian matrix (e.g., see \cite{liuc}). Therefore, an improved tensor splitting method combing with the efficient acceleration technique is still an attractive topic.
	\par
	Anderson acceleration is a powerful  extrapolation technique to improve the fixed-point iteration for solving the linear or nonlinear systems (e.g., see \cite{lai,aa,aa0}). Niu et al. \cite{du} gave an Anderson-Richardson  method for solving the equation \eqref{ax=b}. Though the given methods in \cite{du} perform well, it is a pity that there is no theoretical analysis yet to ensure the nonnegativity of the iteration when the multilinear systems are solved. By this motivation, we propose a tensor splitting method with the relaxed Anderson acceleration and the nonnegativity is preserved in every new iterative step for solving the multilinear systems with the coefficient tensor being a strong $\mathcal{M}$-tensor. Besides, the convergence analysis is given for preserving the nonnegativity. Furthermore, the proposed algorithm performs better than the existing preconditioned methods in numerical experiments.
	

	The rest of this paper is given as follows. We introduce some definitions and lemmas in Section \ref{pre}. Besides, the tensor splitting method with the relaxed Anderson acceleration is proposed in Section \ref{als}. Furthermore, the convergence analysis is considered in Section \ref{conana}. We report numerical experiments to present the efficient performance of the proposed method in Section \ref{numex}. The final section is the conclusion.
	
	\section{Preliminaries}\label{pre}
	{An order $l$ and dimension $n$ tensor $\mathcal{A}$ with $n^{l}$ elements is defined by}	
	{\small{$$\mathcal{A} = ( a_{ i_1 \cdots i_l } ), a_{ i_1 \cdots i_l } \in \mathbb{R}, i_j \in \langle n\rangle, j=1,\cdots ,l,$$}}
where $\mathbb{R}$ denotes the real field, $\langle n \rangle = \{ 1 , \cdots , n \}$ and $n$ is a positive integer. We use $\mathbb{R}^{[l,n]}$ to denote the set of all order $l$ dimension $n$ real tensors. If $l=1$, $\mathbb{R}^{[1,n]}$ {reduces to} $\mathbb{R}^{n}$, which is commonly used to denote the set of all dimension $n$ real vectors. More specifically, $\mathbb{R}_{+}^{n}$ ($\mathbb{R}_{++}^{n}$)
	is indicted as the set of all nonnegative (positive) vectors. Let ${{0}}$ denote a null number, a null vector, a null matrix or a null tensor. Let $\mathcal {A},~\mathcal {B}\in\mathbb{R}^{[l,n]}$. The notation $\mathcal {A}\ge \mathcal {B}$ $(>\mathcal {B})$ refers that every element of $\mathcal {A}$ is no less than (strictly  greater than) the corresponding one of $\mathcal {B}$. Furthermore, $||\cdot||$ denotes {the} 2-norm of a vector or a matrix.

	Qi \cite{Qi3} and Lim \cite{lim}  independently gave the definition of tensor eigenvalues and eigenvectors, i.e., a complex number $\lambda$  and a   nonzero complex vector $ {\bf{x}}$ are called as an eigenvalue and an eigenvector 	of $\mathcal{A}\in \mathbb{R}^{[l,n]}$, respectively, if the following equation holds:
	\begin{equation*}
		\mathcal{A}{\bf{x}}^{l-1}=\lambda {\bf{x}}^{[l-1]},
	\end{equation*}
	where ${\bf{x}}^{[l-1]}=(x_1^{l-1},...,x_n^{l-1})^{T}$. Besides, $(\lambda ,{\bf{x}})$ is called as an \textit{$H$-eigenpair} if both $\lambda $ and ${\bf{x}}$ are real. Furthermore, the spectral radius of $\mathcal {A}$ is signified by $\rho (\mathcal{A})=\max \{\left| \lambda \right| |\lambda \in \sigma (%
	\mathcal{A})\}$ with $\sigma (\mathcal{A})$ being the set of all eigenvalues of $\mathcal {A}$. We call $\mathcal{I}_{l}\in \mathbb{R}^{[l,n]}$ a \textit{unit tensor} if its entries are given {by}
	$${\delta _{{i_1} \cdots {i_l}}} = \left\{ \begin{array}{l}
		1,{\kern 1pt} {\kern 1pt} {\kern 1pt} {\kern 1pt} {\kern 1pt} {\kern 1pt} {\kern 1pt} {\kern 1pt} {\kern 1pt} {\kern 1pt} {\kern 1pt} {\kern 1pt} {\kern 1pt} {\kern 1pt} {\kern 1pt} {\kern 1pt} {i_1} =  \cdots  = {i_l}, \\
		0,{\kern 1pt} {\kern 1pt} {\kern 1pt} {\kern 1pt} {\kern 1pt} {\kern 1pt} {\kern 1pt} {\kern 1pt} {\kern 1pt} {\kern 1pt} {\kern 1pt} {\kern 1pt} {\kern 1pt} {\kern 1pt} \mbox{else.} \\
	\end{array} \right.$$
	 {$\mathcal {A}\in\mathbb{R}^{[l,n]}$} is a \textit{$\mathcal {Z}$-tensor} if its off-diagonal entries are non-positive. Besides, $\mathcal {A}$ is an \textit{$\mathcal {M}$-tensor} (see the works of \cite{wei1, Qi2}) if there exists a nonnegative tensor $\mathcal {B}$ and a positive real number $\eta  \ge \rho (\mathcal {B})$ such that $\mathcal {A}=\eta \mathcal {I}_{l}  -  \mathcal {B}.$
	Furthermore, if $\eta  > \rho (\mathcal {B})$, $\mathcal {A}$ is  a \textit{strong $\mathcal {M}$-tensor}. Let  $A\in {\mathbb{R}^{[2,n]}}$ and $\mathcal {B} \in {\mathbb{R}^{[k,n]}}$.  A product \cite{dec}  $\mathcal {C}=A\mathcal {B}\in \mathbb{R}^{[k,n]}$ is introduced by
	\begin{equation*}
		{c_{j{i_2} \cdots {i_k}}} = \sum\limits_{{j_2}=1}^{{n}} {{a_{j{j_2}}}{b_{{j_2}{i_2} \cdots {i_k}}}}.
	\end{equation*}
	
 In \cite{llh}, it is pointed out that if $\mathcal{A}$ is a left-nonsingular tensor if and only if $\mathcal{A}=M(\mathcal{A})\mathcal{I}_{l}$ and $M(\mathcal{A})$ is nonsingular. If $\mathcal{A}$ is a strong $\mathcal {M}$-tensor, then $M(\mathcal {A})$ is a nonsingular $\mathcal {M}$-matrix. Besides, the some concepts on tensor splittings are listed as follows.
	\begin{definition}\label{myspitting}\cite{llh}
		Let $\mathcal {A},{\kern 2pt}\mathcal{E},{\kern 2pt}\mathcal{F}\in \mathbb{R}^{[l,n]}$.
		\begin{description}
			\item[(1)] $\mathcal {A}=\mathcal {E}-\mathcal {F}$ is a splitting of $\mathcal {A}$ if and only if $\mathcal {E}$ is a left-nonsingular.
			\item[(2)] $\mathcal {A}=\mathcal {E}-\mathcal {F}$ is a regular splitting of $\mathcal {A}$ if and only if $\mathcal {E}$ is left-nonsingular with $M(\mathcal {E})^{-1}\ge 0$, and $\mathcal {F}\ge 0$.
			\item[(3)]  $\mathcal {A}=\mathcal {E}-\mathcal {F}$ is a weak regular splitting of $\mathcal {A}$ if and only if $\mathcal {E}$ is left-nonsingular with $M(\mathcal {E})^{-1}\ge 0$ and $M(\mathcal {E})^{-1}\mathcal {F}\ge 0$.
			\item[(4)] $\mathcal {A}=\mathcal {E}-\mathcal {F}$ is a convergent splitting if and only if $\rho (M(\mathcal {E})^{{ - 1}}\mathcal {F}) < 1$.
		\end{description}
		
	\end{definition}
	\begin{remark}
	{{ From the works of \cite{llh, Qi2},
if $\mathcal {A}$ is a $\mathcal {Z}$-tensor, the following terms are equivalent:
\begin{description}
	\item[(1)] $\mathcal {A}$ is a strong $\mathcal {M}$-tensor.
	\item[(2)] $\mathcal {A}$ has a convergent (weak) regular splitting.
	\item[(3)] All the (weak) regular splittings of $\mathcal {A}$ are convergent.
		\item[(4)] All the $H$-eigenvalue of $\mathcal {A}$ are positive.
			\item[(5)] All the real part of each eigenvalue of $\mathcal {A}$ is positive.
\end{description}}}

	\end{remark}
	{For more spectral structure and property of a strong $\mathcal{M}$-tensor, please refer to the works of \cite{Qi2}. The following lemmas on multilinear systems in the existing works will be used in Section \ref{conana}.}
	\begin{lemma}\label{re1}\cite{wei2}
		Let $\mathcal {A}$ be a strong $\mathcal {M}$-tensor and ${\bf{b}}\in \mathbb{R}^{n}_{++}$. Thus, the multilinear systems \eqref{ax=b} have a unique positive solution.
	\end{lemma}
	\begin{lemma}\cite{llh}\label{invmethid}
		Let $\mathcal {A}$ be a strong $\mathcal {M}$-tensor and ${\bf{b}}\in\mathbb{R}^{n}_{++}$. Thus, there always exists a ${\bf{x}}_{0}\in \mathbb{R}^{n}_{++}$ with ${\bf{0}}<\mathcal {A}{{\bf{x}}_{0}^{l-1}}\le {\bf{b}}$.
	\end{lemma}
	\section{Tensor splitting methods with Anderson acceleration}\label{als}
	Let $f_{\mathcal{E}}({\bf{x}})=g_{\mathcal{E}}({\bf{x}})-{\bf{x}}.$ The tensor splitting methods with the relaxed Anderson acceleration is proposed  for solving the multilinear systems \eqref{ax=b} as follows.
	\begin{breakablealgorithm}\label{TSMAA}
		\caption{Restarted tensor splitting methods with the relaxed Anderson acceleration(RTSMRAA)}
		\begin{algorithmic}[1]
			\STATE{Require: Given a positive vector ${\bf{b}}$, a strong $\mathcal {M}$-tensor $\mathcal {A}$ with its (weak) regular splitting $\mathcal {A}=\mathcal {E}-\mathcal {F}$, maximum $k_{{\bf{max}}}$, a positive initial vector ${\bf{z}}_0$ and $m\ge 1$.}
			\STATE{Update ${\bf{z}}_1=g_{\mathcal{E}}({\bf{z}}_0)$, ${\boldsymbol{\mu}}_1={\bf{z}}_1$,  ${\bf{y}}_1={\bf{z}}_1$ and $f_0={\bf{z}}_1-{\bf{z}}_0$}
			\FOR{k=1,2,..., $k_{{\bf{max}}}$}
			\STATE{Set $m_k=\min\{m,k\}$}
			\STATE{Update
				\begin{equation}\label{it1}
					\begin{split}
						{\boldsymbol{\mu}}_{k+1}&=g_{\mathcal{E}}({\bf{z}}_k),\\
						f_k&={\boldsymbol{\mu}}_{k+1}-{\bf{z}}_k.
					\end{split}
				\end{equation}
			}
			\STATE{Set $Q_k=(f_{k-m_k},...,f_k)$.}
			\STATE{Determine $\boldsymbol{\alpha}^{(k)}=(\alpha_{0}^{(k)},...,\alpha_{m_k}^{(k)})^{T}$ that solves
				\begin{equation}\label{sq2}
					\min_{\alpha=(\alpha_0,...,\alpha_{m_k})^{T}} {||Q_k\alpha||}~\textit{s.t.}~ \sum\limits_{i=0}^{m_k}\alpha_i =1.
				\end{equation}	
				\STATE{Update
					$$	{\bf{y}}_{k+1}=\sum_{i=0}^{m_k}\alpha_{i}^{(k)}g_{\mathcal{E}}(z_{k-m_k+i}).$$
				}
			}				
			\IF{${\bf{y}}_{k+1}\ge 0$ and $\sum\limits_{i=0}^{m_k}|\alpha_i^{(k)}|\le \kappa_{\boldsymbol{\alpha}}$}
			\STATE{Choose $\theta_{k}\in[0,1]$ to update ${\bf{z}}_{k+1}=\theta_k{\bf{y}}_{k+1}+(1-\theta_k){{\boldsymbol{\mu}}}_{k+1}$.}
			\ELSE
			\STATE{Update ${\bf{z}}_{k+1}=\boldsymbol{\mu}_{k+1}$.}
			\ENDIF	
			\STATE{Update $k\leftarrow k+1$.}
			\ENDFOR
			\RETURN ${\bf{z}}_{k}$
		\end{algorithmic}
	\end{breakablealgorithm}
	\begin{remark}\label{re}
		The optimization problem \eqref{sq2} can be solved by the following equivalent the square problem\cite{aa0}:
		Give $\boldsymbol{\zeta}^{(k)}=(\zeta_0^{(k)},\zeta_1^{(k)},...,\zeta_{m_k-1}^{(k)})^{T}$ by solving
		\begin{equation}\label{sq}
			\min_{\boldsymbol{\zeta}^{(k)}=(\zeta_0^{(k)},\zeta_1^{(k)},...,\zeta_{m_k-1}^{(k)})^{T}}{||f_k- F_k\boldsymbol{\zeta}^{(k)}||},
		\end{equation}
		where $ F_{k}=(\Delta f_{k-m_k},...,\Delta f_{k-1})$ and $\Delta f_i=f_{i+1}-f_i$.
		Furthermore, Step 8 in the RTSMRAA is written as follows:
		\begin{equation}\label{upaa}
			{\bf{y}}_{k+1}={\boldsymbol{\mu}}_{k+1}- G_{k}\boldsymbol{\zeta}^{(k)},
		\end{equation}
		where $G_{k}=(\Delta g_{k-m_k},...,\Delta g_{k-1})$ and $\Delta g_i=g_{\mathcal{E}}({\bf{z}}_{i+1})-g_{\mathcal{E}}({\bf{z}}_{i}).$
		The solution $\boldsymbol{\zeta}$ of the least square problem \eqref{sq} and the solution $\alpha$ of the optimization problem \eqref{sq2} are related by $\alpha_{0}=\zeta_0, \alpha_i=\zeta_i-\zeta_{i-1}$ for $i=1,...,m_k-1$ and $\alpha_{m_k}=1-\zeta_{m_k-1}$.
	\end{remark}

	\section{Convergence analysis}\label{conana}

	In this section, we will consider the convergence analysis of Algorithm \ref{TSMAA}. Let $\mathbb{U}_{\delta}\equiv\{{\bf{x}}\big{|}||{\bf{x}}-{\bf{x}}^{*}||\le \delta,\delta>0\}$. We write $a_{i_1i_2\cdots i_l}$ as $a_{i_1\pi}$, where $\pi=i_2\cdots i_l$.
	Notice that for any $\mathcal{A}=(a_{i_1i_2\cdots i_l})\in{\mathbb{R}}^{[l,n]}$, it is easy to check that
	\begin{equation*}
		{\mathcal {A}}{\bf{x}}^{l-1}=\overline{\mathcal {A}}{\bf{x}}^{l-1}
	\end{equation*}
	where $\overline{\mathcal{A}}=(\overline{a}_{i_1i_2\cdots i_l})$ is given by:
	\begin{equation*}
		{\overline{a}_{ii_2\cdots i_l}} = \frac{1}{(l - 1)!}\sum\limits_{\pi \in \mathbb{G}} {{a_{i\pi}}},
	\end{equation*}
	and $\mathbb{G}$ denotes the set of all permutations in $({i_2} \cdots {i_l})$. Then $\overline{\mathcal{A}}$ is symmetric on the last $l-1$ indices ($\overline{\mathcal{A}}$ is called semi-symmetric\cite{semi}). For the sake of clarity, we always denote the semi-symmetric of a tensor by the notation `$\overline{{\cdot}}$' (e.g., $\overline{\mathcal{A}}$). Firstly, we give some lemmas as follows.
	\begin{lemma}\label{lemma0}
		Let  $\mathcal {A}$ be a strong $\mathcal {M}$-tensor, ${\bf{b}}\in \mathbb{R}^{n}_{++}$ and $\mathcal {A}=\mathcal {E}-\mathcal {F}$ be a (weak) regular splitting. Then the function $g_{\mathcal{E}}({\bf{x}})$  has the following properties:
		\begin{description}
			\item[(1)] $g_{\mathcal{E}}({\bf{x}})\in \mathbb{R}^{n}_{++}$ for all ${\bf{x}}\in  \mathbb{R}^{n}_{+}$.
			\item[(2)] $g_{\mathcal{E}}({\bf{x}})$  is  monotone increasing in $\mathbb{R}^{n}_{+}$, i.e., for any $0\le{\bf{x}}\le{\bf{y}}$, $g_{\mathcal{E}}({\bf{x}})\le g_{\mathcal{E}}({\bf{y}})$.
		\end{description}
	\end{lemma}
	\begin{proof}
		Since $\mathcal{A}=\mathcal{E}-\mathcal{F}$ is a (weak) regular splitting, the second assertion is obvious. Because ${\bf{b}}\in \mathbb{R}^{n}_{++}$, it is obtained that $M(\mathcal{E})^{-1}{\bf{b}}\in \mathbb{R}^{n}_{++}$. Thus, the first affirmation is proved.
	\end{proof}
	
	
	\begin{proposition}\label{lemma2}
		Under the assumption of Lemma \ref{lemma0}, for any initial vector ${\bf{z}}_0\in \mathbb{R}^{n}_{++}$ with ${{0}}<\mathcal {A}{{\bf{z}}_{0}^{l-1}}\le {\bf{b}}$, we get ${\bf{z}}_{k}\in \mathbb{R}^{n}_{++}$ generated by the {RTSMRAA} if $\theta_{k}\in[0,1)$ in Step 10 of the RTSMRAA for all $k\ge 1$.
	\end{proposition}
	\begin{proof}
		Since $\mathcal{A}{\bf{z}}_0^{l-1}\le {\bf{b}}$, it is true that ${\bf{z}}_{0}\le {\bf{z}}_1$ under the assumption of Lemma \ref{lemma0}.
		By Lemma \ref{lemma0}, we know $\boldsymbol{\mu}_{k}\in \mathbb{R}^{n}_{++}$ for any $k\ge 1$ via the induction. Then the proof is completed.
	\end{proof}
	\begin{remark}\label{re2}
		
		If $\theta_k=1$ in Step 10 of the RTSMRAA for some $k$, it is noted that the TSM is just improved by the restarted Anderson acceleration for the $k$-th iteration, i.e., ${\bf{z}}_{k}={\bf{y}}_k$. In this case,  we still set ${\bf{y}}_{k+1}\ge 0$ rather than ${\bf{y}}_{k+1}>0$ in Step 9 of the RTSMRAA. That's because, by Lemma \ref{lemma0}, we can still get a vector ${\bf{z}}_{k+1}\in \mathbb{R}^{n}_{++}$. If ${\bf{y}}_{k+1}$ is not nonnegative for all $k$, the RTSMRAA reduces to the TSM. Therefore, the next step ${\bf{z}}_{k+1}$ can be always positive by Lemma \ref{lemma0}.

	\end{remark}

	\begin{lemma}\label{map}
		Let ${\bf{x}}^{*}$ be the solution of the multilinear systems \eqref{ax=b} with $\mathcal {A}$ being a strong $\mathcal {M}$-tensor and ${\bf{b}}\in \mathbb{R}^{n}_{++}$.  For every (weak) regular splitting $\mathcal {A}=\mathcal {E}-\mathcal {F}$, there exists a $\mathbb{U}_{\delta}$ such that the mapping $g_{\mathcal{E}}({\bf{x}}):\mathbb{R}^{n}\rightarrow \mathbb{R}^{n}$ has the following properties for a sufficiently small $\delta$:
		\begin{description}
			\item[(\romannumeral1)] For any ${\bf{x}}\in \mathbb{U}_{\delta}$, $g_{\mathcal{E}}({\bf{x}})\in \mathbb{U}_{\delta}$.
			\item[(\romannumeral2)] For any ${\bf{x}}$, ${\bf{y}}\in \mathbb{U}_{\delta}$, there exists a real number $c\in [0,1)$ such that	
			\begin{equation*}
				||g_{\mathcal{E}}({\bf{x}})-g_{\mathcal{E}}({\bf{y}})||\le c||{\bf{x}}-{\bf{y}}||.
			\end{equation*}
			\item[(\romannumeral3)] For any ${\bf{x}}\in \mathbb{U}_{\delta}$,
			\begin{equation*}
				\nabla g_{\mathcal{E}}({\bf{x}})=\big(\mathcal{E}g_{\mathcal{E}}({\bf{x}})^{[l-2]}\big)^{-1}\mathcal{\overline{F}}{\bf{x}}^{l-2},~\nabla f_{\mathcal{E}}({\bf{x}})=\nabla g_{\mathcal{E}}({\bf{x}})-I.
			\end{equation*}
			\item[(\romannumeral4)] $\nabla g_{\mathcal{E}}({\bf{x}}^{*})$ is Lipschitz continuously, i.e., for any ${\bf{x}}$, ${\bf{y}}\in \mathbb{U}_{\delta}$,  there exists a real number $\gamma>0$ such that
			\begin{equation*}
				||\nabla f_{\mathcal{E}}({\bf{x}})-\nabla f_{\mathcal{E}}({\bf{y}})||\le \gamma ||{\bf{x}}-{\bf{y}}||.
			\end{equation*}
		\end{description}
		
	\end{lemma}
	\begin{proof}
		Since $\mathcal {A}$ is a strong $\mathcal {M}$-tensor and ${\bf{b}}$ is positive, by Lemma \ref{re1}, ${\bf{x}^{*}}$ is the unique  positive solution. By the proof of Theorem 5.7 in \cite{llh}, we have that $\mathcal{E}g_{\mathcal{E}}({\bf{x}^{*}})^{[l-2]}=\mathcal{E}({\bf{x}^{*}})^{[l-2]}$ is nonsingular and  $\nabla g_{\mathcal{E}}({\bf{x}})=\big(\mathcal{E}g_{\mathcal{E}}({\bf{x}})^{[l-2]}\big)^{-1}\mathcal{\overline{F}}{\bf{x}}^{l-2}$ satisfying {$\rho(\nabla g_{\mathcal{E}}({\bf{x}}^{*}))\le c<1$}.  Because $g_{\mathcal{E}}({\bf{x}})$ and $\nabla g_{\mathcal{E}}({\bf{x}})$ are continuous, there exists $\mathbb{U}_{\delta}$ such that {$\rho\big(\nabla g_{\mathcal{E}}({\bf{x}})\big)\le c<1$} and  $\mathcal{E}g_{\mathcal{E}}({\bf{x}})^{[l-2]}$ is nonsingular for any ${\bf{x}}\in \mathbb{U}_{\delta}$. Thus, it is easy to check that the affirmation \textbf{(\romannumeral2)} holds by differential mean value theorem and the term \textbf{(\romannumeral3)} is given. Besides, for any ${\bf{x}}\in \mathbb{U}_{\delta}$, we have
		\begin{align*}
			||g_{\mathcal{E}}({\bf{x}})-{\bf{x}}^{*}||=	||g_{\mathcal{E}}({\bf{x}})-g_{\mathcal{E}}({\bf{x}}^{*})||\le c||{\bf{x}}-{\bf{x}}^{*}||<||{\bf{x}}-{\bf{x}}^{*}||\le \delta.
		\end{align*}
		Then the desired result \textbf{(\romannumeral1)} is obtained.
		Furthermore, by the result 3.3.5 of \cite{non} and  the second G-derivative $\nabla^2 f_{\mathcal{E}}({\bf{x}})$ of $\nabla f_{\mathcal{E}}({\bf{x}})$ at ${\bf{x}}$ is continuous, we have
		\begin{equation*}
			||\nabla f_{\mathcal{E}}({\bf{x}})-\nabla f_{\mathcal{E}}({\bf{y}})||\le \max\limits_{{\bf{x},{\bf{y}}}\in \mathbb{R}^{n}_{+}}\max\limits_{t\in [0,1]}\big( ||\nabla^2 f_{\mathcal{E}}({\bf{x}}+t({\bf{y}}-{\bf{x}}))||\big)||{\bf{x}}-{\bf{y}}||.
		\end{equation*}
		Since $\mathbb{R}^{n}_{+}$ is a closed set, the term \textbf{(\romannumeral4)} is shown by $\gamma=\max\limits_{{\bf{x},{\bf{y}}}\in \mathbb{R}^{n}_{+}}\max\limits_{t\in [0,1]}\big( ||\nabla^2 f_{\mathcal{E}}({\bf{x}}+t({\bf{y}}-{\bf{x}}))||\big).$
	\end{proof}
	
	\begin{proposition}\label{le2}
		Under the assumption of Lemma \ref{map}, we have $\big(\mathbb{I}\setminus\{{\bf{x}}^{*}\}\big)\cap \big(\mathbb{U}_{\delta}\setminus\{{\bf{x}}^{*}\})\ne \emptyset$, where $\mathbb{I}=\big\{{\bf{x}}\big{|}{\bf{x}}>0, 0<\mathcal{A}{\bf{x}}^{l-1}\le b\big\}$ and $\mathbb{U}_{\delta}$ is defined in Lemma \ref{map}.
	\end{proposition}
	\begin{proof}
		By Lemma \ref{invmethid}, $\mathbb{I}\setminus\{{\bf{x}}^{*}\}\ne \emptyset$. Let  $p({\bf{x}})=\mathcal{A}{\bf{x}}^{l-1}$ which is a continuous function from $\mathbb{R}^{n}\rightarrow \mathbb{R}^{n}$. Thus, there exists a real number $\delta^{'}>0$ such {that $0<\mathcal{A}{\widehat{{\bf{x}}}}^{l-1}\le b$} for any $\widehat{{\bf{x}}}\in \mathbb{U}_{\delta^{'}}\setminus\{{\bf{x}}^{*}\}$. Let $\delta^{''}=\min\big\{\delta^{'},\delta\big\}$. It is easy to check that $\widehat{{\bf{x}}}\in \mathbb{U}_{\delta^{''}}\subseteq\mathbb{U}_{\delta}$. Therefore, the proof is completed.
	\end{proof}
	\begin{lemma}\label{lemain}
		Under the assumption of Lemma \ref{map}, there exists $\mathbb{U}_{\hat{\delta}}\subseteq\mathbb{U}_{\delta}$ with sufficiently small $\hat{\delta}$ such that for any ${\bf{x}}\in \mathbb{R}^{n}_{++}$,
		\begin{equation}\label{eq}
			||f_{\mathcal{E}}({\bf{x}})-\nabla f_{\mathcal{E}}({\bf{x}}^{*})({\bf{x}}-{\bf{x}}^{*})||\le \frac{\gamma}{2}||{\bf{x}}-{\bf{x}}^{*}||^2
		\end{equation}
		and
		\begin{equation}\label{eq2}
			(1-c)||{\bf{x}}-{\bf{x}}^{*}||\le||f_{\mathcal{E}}({\bf{x}})||\le(1+c)||{\bf{x}}-{\bf{x}}^{*}||.
		\end{equation}
	\end{lemma}
	\begin{proof}
		The formula \eqref{eq} is a direct result by 3.2.12 of \cite{non} and \textbf{(\romannumeral3)} of Lemma \ref{map}. The second formula can be {shown} by \textbf{(\romannumeral2)} of Lemma \ref{map} and
		\begin{align*}
			||f_{\mathcal{E}}({\bf{x}})||=||f_{\mathcal{E}}({\bf{x}})-f_{\mathcal{E}}({\bf{x}}^{*})||=||g_{\mathcal{E}}({\bf{x}})-g_{\mathcal{E}}({\bf{x}}^{*})+({\bf{x}}-{\bf{x}}^{*})||.
		\end{align*}
	\end{proof}

	Let $\boldsymbol{\epsilon}_k={\bf{z}}_k-{\bf{z}}^{*}$,  $\hat{\boldsymbol{\epsilon}}_k={\bf{y}}_k-{\bf{z}}^{*}$and ${\bf{z}}^{*}={\bf{x}}^{*}.$ Then we give the convergence analysis of the proposed algorithm as follows.
	
	\begin{theorem}\label{add0}
		Under the assumption of Lemma \ref{map}, for any initial vector ${\bf{z}}_{0}\in \mathbb{I}\cap \mathbb{U}_{\delta}$, there exists $\mathbb{U}_{\hat{\delta}}\subseteq \mathbb{I}\cap \mathbb{U}_{\delta}$ with a small enough $\hat{\delta}$ such that ${\bf{y}}_k,~{\bf{z}}_{k}\in \mathbb{U}_{\hat{\delta}}$($k\ge 1$), and
		\begin{equation}\label{fe}
			||f_{\mathcal{E}}({\bf{z}}_k)||\le\tau \overline {c}^{k}||f_{\mathcal{E}}({\bf{z}}_0)||,
		\end{equation}
		where {${\bf{y}}_k$ and ${\bf{z}}_k$ are} generated by the RTSMRAA, $\forall \overline{c}\in (c,1)$ and $\tau=\frac{1+c}{1-c}$.
	\end{theorem}
	\begin{proof}
		We can always choose the ${\bf{z}}_{0}\in \mathbb{U}_{\hat{\delta}}$ with a sufficiently small $\hat{\delta}$ such that $\hat{\delta}<2(1-c)/\gamma$,
		\begin{equation}\label{e3}
			\frac{\kappa_{\boldsymbol{\alpha}}\tau(2c+\gamma\hat{\delta})}{2(1-c)}\overline{c}^{-m}||f_{\mathcal{E}}({\bf{z}}_{0})||\le 	\frac{\kappa_{\boldsymbol{\alpha}}\tau(2c+\gamma\hat{\delta})(1+c)}{2(1-c)}\overline{c}^{-m}||{\boldsymbol{\epsilon}}_0||\le \hat{\delta}
		\end{equation}	
		and
		\begin{equation}\label{e4}
			\frac{\frac{c}{\overline{c}}+\frac{\kappa_{\boldsymbol{\alpha}}\gamma\hat{\delta}}{2(1-c)}\overline{c}^{-m-1}}{1-\frac{\hat{\delta}\gamma}{2(1-c)}}\le 1.	
		\end{equation}
		
		For any initial vector ${\bf{z}}_{0}\in \mathbb{U}_{\hat{\delta}}\subseteq \mathbb{I}\cap \mathbb{U}_{\delta}$ with a sufficiently small $\hat{\delta}$ such that \eqref{eq} and \eqref{eq2} hold, by Lemma \ref{map} and the proof of Proposition \ref{le2},  we know that there exists $\mathbb{U}_{\hat{\delta}}\subseteq \mathbb{I}\cap \mathbb{U}_{\delta}$ such that ${\bf{z}}_1,~{\bf{y}}_1,~{\boldsymbol{\mu}}_1,~{\boldsymbol{\mu}_2}\in \mathbb{U}_{\hat{\delta}}$ and the inequality \eqref{fe} holds. Assume that ${\bf{y}}_k,~{\bf{z}}_k \in \mathbb{U}_{\hat{\delta}}$ and $||f_{\mathcal{E}}({\bf{z}}_k)||\le \tau\overline {c}^{k}||f_{\mathcal{E}}({\bf{z}}_0)||$ for any $0 \le k\le N$. Now we will prove  ${\bf{z}}_{N+1}\in \mathbb{U}_{\hat{\delta}}$ and the inequality $||f_{\mathcal{E}}({\bf{z}}_{N+1})||\le \tau\overline {c}^{N+1}||f_{\mathcal{E}}({\bf{z}}_0)||$ by the induction. By Lemma \ref{lemain}, we have
		\begin{equation}\label{addeq0}
			f_{\mathcal{E}}({\bf{z}}_k)=\nabla f_{\mathcal{E}}({\bf{z}}^{*}){{\boldsymbol{\epsilon}}_k}+{\varOmega}_k,
		\end{equation}	
		and $||{\varOmega_k}||\le \frac{\gamma}{2}||{\boldsymbol{\epsilon}_k}||^2$.
		Therefore, it is easy to check that
		\begin{equation}\label{addeq00}
			g_{\mathcal{E}}({\bf{z}}_k)={\bf{z}}^{*}+\nabla g_{\mathcal{E}}({\bf{z}}^{*})\boldsymbol{\epsilon}_k+\varOmega_k.
		\end{equation}
		Firstly, we will prove \eqref{fe} and ${\bf{y}}_{N+1}\in \mathbb{U}_{\hat{\delta}}$ from the following two cases.\\
		
		{\bf{Case \uppercase\expandafter{\romannumeral1}}}: If ${\bf{y}}_{N+1}\le 0$ or $\sum\limits_{i=0}^{m_k}|\alpha_i^{(k)}|> \kappa_{\boldsymbol{\alpha}}$, by Lemma \ref{map}, Lemma \ref{lemain} and the hypothesis, we imply ${\bf{z}}_{N+1}=\boldsymbol{\mu}_{N+1}\in \mathbb{U}_{\hat{\delta}}$ and
		\begin{align*}
			||f_{\mathcal{E}}({\bf{z}}_{N+1})||&\le (1+c)||{\boldsymbol{\epsilon}}_{N+1}||\\
			&\le(1+c)||g_{\mathcal{E}}({\bf{z}}_{N})-g_{\mathcal{E}}({\bf{z}}^{*})|| \\
			&\le c(1+c)||\boldsymbol{\epsilon}_{N}||\\
			&\le \frac{c(1+c)}{1-c}||f_{\mathcal{E}}({\bf{z}}_N)||\\
			&\le \frac{1+c}{1-c}\overline{c}^{N+1}||f_{\mathcal{E}}({\bf{z}}_0)||=\tau\overline{c}^{N+1}||f_{\mathcal{E}}({\bf{z}}_0)||.
		\end{align*}
		This gives the inequality \eqref{fe}.
		
		{\bf{Case \uppercase\expandafter{\romannumeral2}}}: If ${\bf{y}}_{N+1}\ge 0$ and $\sum\limits_{i=0}^{m_k}|\alpha_i^{(k)}|\le \kappa_{\boldsymbol{\alpha}}$,
		note that the fact $\sum\limits_{i=0}^{m_{N}}\alpha_i^{(N)}=1$ combing with \eqref{addeq00} gives
		\begin{equation}\label{addeq01}
			{\bf{y}}_{N+1}={\bf{z}}^{*}+\sum\limits_{i=0}^{m_{N}}\alpha_{i}^{(N)}\nabla g_{\mathcal{E}}({\bf{z}}^{*})\boldsymbol{\epsilon}_{N-m_N+i}+\sum\limits_{i=0}^{m_{N}}\alpha_{i}^{(N)}\varOmega_{N-m_N+i}.
		\end{equation}
		By Lemma \ref{lemain}, we can obtain
		\begin{equation}\label{addeq1}
			||\sum\limits_{i=0}^{m_{N}}\alpha_{i}^{(N)}\varOmega_{N-m_N+i}||\le \sum\limits_{i=0}^{m_{N}}|\alpha_{i}^{(N)}|||\varOmega_{N-m_N+i}||\le \gamma\sum\limits_{i=0}^{m_{N}}|\alpha_{i}^{(N)}| ||\boldsymbol{\epsilon}_{N-m_N+i}||^2/2.
		\end{equation}	
		Since $m_{N}=\min\{m,N\}$, we get $N-m_N+i\ge N-m$. Thus, by Lemma \ref{lemain} and the assumption, it is not difficult to get that
		
		\begin{align}\label{e1}
			||\boldsymbol{\epsilon}_{N-m_N+i}||\le \frac{1}{1-c}||f_{\mathcal{E}}({\bf{z}}_{N-m_N+i})||\le \frac{\tau }{1-c}\overline{c}^{N-m_N+i}||f_{\mathcal{E}}({\bf{z}}_{0})||\le \frac{\tau }{1-c}\overline{c}^{-m}||f_{\mathcal{E}}({\bf{z}}_{0})||
		\end{align}
		and
		\begin{align}\label{e2}
			\begin{split}
				||\boldsymbol{\epsilon}_{N-m_N+i}||^2 &\le\frac{1}{1-c} ||\boldsymbol{\epsilon}_{N-m_N+i}||||f_{\mathcal{E}}({\bf{z}}_{N-m_N+i})||\\
				&\le\frac{\hat{\delta}\tau}{1-c}\overline{c}^{N-m_N+i}||f_{\mathcal{E}}({\bf{z}}_{0})||\\
				& \le\frac{\hat{\delta}\tau}{1-c}\overline{c}^{N-m}||f_{\mathcal{E}}({\bf{z}}_{0})||,
			\end{split}
		\end{align}
		together with $\sum\limits_{i=0}^{m_N}|\alpha_i^{(N)}|\le \kappa_{\boldsymbol{\alpha}}$, the inequalities \eqref{addeq1} and \eqref{e1},  we imply that
		\begin{equation}\label{addeq2}
			||\sum\limits_{i=0}^{m_{N}}\alpha_{i}^{(N)}\varOmega_{N-m_N+i}||\le \frac{\kappa_{\boldsymbol{\alpha}}\gamma \hat{\delta}\tau}{2(1-c)}\overline{c}^{N-m}||f_{\mathcal{E}}({\bf{z}}_{0})||\le \frac{\kappa_{\boldsymbol{\alpha}}\gamma \hat{\delta}\tau}{2(1-c)}\overline{c}^{-m}||f_{\mathcal{E}}({\bf{z}}_{0})||
		\end{equation}
		and
		\begin{equation}\label{addeq3}
			||\sum\limits_{i=0}^{m_{N}}\alpha_{i}^{(N)}\nabla g_{\mathcal{E}}({\bf{z}}^{*})\boldsymbol{\epsilon}_{N-m_N+i}||\le \frac{\kappa_{\boldsymbol{\alpha}}\tau c}{1-c}\overline{c}^{-m}||f_{\mathcal{E}}({\bf{z}}_{0})||.
		\end{equation}
		From \eqref{addeq01},  we obtain
		\begin{equation*}
			{\hat{\boldsymbol{\epsilon}}_{N+1}}=\sum\limits_{i=0}^{m_{N}}\alpha_{i}^{(N)}\nabla g_{\mathcal{E}}({\bf{z}}^{*})\boldsymbol{\epsilon}_{N-m_N+i}+\sum\limits_{i=0}^{m_{N}}\alpha_{i}^{(N)}\varOmega_{N-m_N+i},
		\end{equation*}
		combing with \eqref{addeq2}, \eqref{addeq3}  and the condition \eqref{e3} yields
		\begin{align}\label{addeq4}
			||\hat{\boldsymbol{\epsilon}}_{N+1}||\le \frac{\kappa_{\boldsymbol{\alpha}}\tau(2c+\gamma\hat{\delta})}{2(1-c)}\overline{c}^{-m}||f_{\mathcal{E}}({\bf{z}}_{0})||\le \hat{\delta},
		\end{align}
		and hence ${\bf{y}}_{N+1}\in \mathbb{U}_{\hat{\delta}}$. Because ${\boldsymbol{\mu}_{N+1}}\in \mathbb{U}_{\hat{\delta}}$ by Lemma \ref{map}, it is found that ${\bf{z}}_{N+1}\in \mathbb{U}_{\hat{\delta}}.$
		Therefore, by Lemma \ref{lemain}, we can give the Taylor's formula of $f_{\mathcal{E}}({\bf{z}})$ at the true solution ${\bf{z}}^{*}$ as follows.
		\begin{equation}\label{add}
			f_{\mathcal{E}}({\bf{z}}_{N+1})=(\nabla g_{\mathcal{E}}({\bf{z}}^{*})-I){\boldsymbol{\epsilon}_{N+1}}+\varOmega_{N+1},
		\end{equation}
		with $||\varOmega_{N+1}||\le \frac{\gamma}{2}||\boldsymbol{\epsilon}_{N+1}||^2$.
		This combing with \eqref{addeq01} and the fact $\nabla g_{\mathcal{E}}({\bf{z}}^{*})(\nabla g_{\mathcal{E}}({\bf{z}}^{*})-I)=(\nabla g_{\mathcal{E}}({\bf{z}}^{*})-I)\nabla g_{\mathcal{E}}({\bf{z}}^{*})$ yields
		
		\begin{align}\label{allq}
			\begin{split}
				f_{\mathcal{E}}({\bf{z}}_{N+1})=&(\nabla g_{\mathcal{E}}({\bf{z}}^{*})-I)\big(\theta_k\hat{\boldsymbol{\epsilon}}_{N+1}+(1-\theta_k)( g_{\mathcal{E}}({\bf{z}}_N)-{\bf{z}}^{*})\big)+\varOmega_{N+1}\\
				=&\theta_k \nabla g_{\mathcal{E}}({\bf{z}}^{*})\sum\limits_{i=0}^{m_N}\alpha_{i}^{(N)}(\nabla g_{\mathcal{E}}({\bf{z}}^{*})-I){\boldsymbol{\epsilon}}_{N-m_N+i}+\theta_k(\nabla g_{\mathcal{E}}({\bf{z}}^{*})-I)\sum\limits_{i=0}^{m_{N}}\alpha_{i}^{(N)}\varOmega_{N-m_N+i}\\
				&+(1-\theta_k)(\nabla g_{\mathcal{E}}({\bf{z}}^{*})-I)(g_{\mathcal{E}}({\bf{z}}_N)-{\bf{z}}^{*})+\varOmega_{N+1}.
			\end{split}
		\end{align}
		By \eqref{addeq1}-\eqref{addeq2}, \eqref{add} the inductive assumption and the fact $||\sum\limits_{i=0}^{m_N}\alpha_{i}^{(N)}f_{\mathcal{E}}({\bf{z}}_{N-m_N+i})||\le ||f_{\mathcal{E}}({\bf{z}}_{N})||$, it is obtained that
		\begin{align}\label{allq1}
			\begin{split}
				& ||\nabla g_{\mathcal{E}}({\bf{z}}^{*})\sum\limits_{i=0}^{m_N}\alpha_{i}^{(N)}(\nabla g_{\mathcal{E}}({\bf{z}}^{*})-I){\boldsymbol{\epsilon}}_{N-m_N+i}+(\nabla g_{\mathcal{E}}({\bf{z}}^{*})-I)\sum\limits_{i=0}^{m_{N}}\alpha_{i}^{(N)}\varOmega_{N-m_N+i}||\\
				=&||\nabla g_{\mathcal{E}}({\bf{z}}^{*})\sum\limits_{i=0}^{m_N}\alpha_{i}^{(N)}\big(f_{\mathcal{E}}({\bf{z}}_{N-m_N+i})-\varOmega_{N-m_N+i}\big)+(\nabla g_{\mathcal{E}}({\bf{z}}^{*})-I)\sum\limits_{i=0}^{m_{N}}\alpha_{i}^{(N)}\varOmega_{N-m_N+i}||\\
				=&||\nabla g_{\mathcal{E}}({\bf{z}}^{*})\sum\limits_{i=0}^{m_N}\alpha_{i}^{(N)}f_{\mathcal{E}}({\bf{z}}_{N-m_N+i})-\sum\limits_{i=0}^{m_{N}}\alpha_{i}^{(N)}\varOmega_{N-m_N+i}||\\
				\le&c||f_{\mathcal{E}}({\bf{z}}_{N})||+||\sum\limits_{i=0}^{m_{N}}\alpha_{i}^{(N)}\varOmega_{N-m_N+i}||\\
				\le&(\frac{c}{\overline{c}}+\frac{\kappa_{\boldsymbol{\alpha}}\gamma \hat{\delta}}{2(1-c)}\overline{c}^{-m-1})\tau\overline{c}^{N+1}||f_{\mathcal{E}}({\bf{z}}_{0})||
			\end{split}
		\end{align}
		and
		\begin{align}\label{allq2}
			\begin{split}
				|| (\nabla g_{\mathcal{E}}({\bf{z}}^{*})-I)(g_{\mathcal{E}}({\bf{z}}_N)-{\bf{z}}^{*})||
				=&||(\nabla g_{\mathcal{E}}({\bf{z}}^{*})-I)f_{\mathcal{E}}({\bf{z}}_N)+(\nabla g_{\mathcal{E}}({\bf{z}}^{*})-I)
				\boldsymbol{\epsilon}_{N}||\\
				\le&||\nabla g_{\mathcal{E}}({\bf{z}}^{*})f_{\mathcal{E}}({\bf{z}}_N)||+||\varOmega_{N}||\\
				\le&(\frac{c}{\overline{c}}+\frac{\gamma\hat{\delta}}{2(1-c)}\overline{c}^{-m-1})\tau\overline{c}^{N+1}||f_{\mathcal{E}}({\bf{z}}_{0})||.
			\end{split}
		\end{align}	
		Note that $||\varOmega_{N+1}||\le \frac{\gamma\hat{\delta}}{2(1-c)}||f_{\mathcal{E}}({\bf{z}}_{N+1})||$. By \eqref{allq}-\eqref{allq2}{,} the hypothesis \eqref{e4} and the fact $\kappa_{\boldsymbol{\alpha}}\ge 1$, it is not {difficult} to get that
		\begin{align}\label{allq3}
			\begin{split}
				||f_{\mathcal{E}}({\bf{z}}_{N+1})||&\le \theta_k\frac{\frac{c}{\overline{c}}+\frac{\kappa_{\boldsymbol{\alpha}}\gamma \hat{\delta}}{2(1-c)}\overline{c}^{-m-1}}{1-\frac{\gamma\hat{\delta}}{2(1-c)}}\tau\overline{c}^{N+1}||f_{\mathcal{E}}({\bf{z}}_{0})||+(1-\theta_k)\frac{\frac{c}{\overline{c}}+\frac{\gamma\hat{\delta}}{2(1-c)}\overline{c}^{-m-1}}{1-\frac{\gamma\hat{\delta}}{2(1-c)}}\tau\overline{c}^{N+1}||f_{\mathcal{E}}({\bf{z}}_{0})||\\
				&\le \tau\overline{c}^{N+1}||f_{\mathcal{E}}({\bf{z}}_{0})||.
			\end{split}
		\end{align}
		The proof is completed.
	\end{proof}
	From Lemma \ref{lemain} and Theorem \ref{add0}, we can get the following corollary immediately.
	\begin{corollary}
		Under the assumption of Lemma \ref{map}, for any initial vector ${\bf{z}}_{0}\in \mathbb{I}\cap \mathbb{U}_{\delta}$, there exists $\mathbb{U}_{\hat{\delta}}\subseteq \mathbb{I}\cap \mathbb{U}_{\delta}$ with small enough $\hat{\delta}$ such that
		\begin{equation}\label{fe2}
			||\boldsymbol{\epsilon}_k||\le\tau^2 \overline {c}^{k}||\boldsymbol{\epsilon}_0||,
		\end{equation}
		where $\tau$ and $\overline{c}$ {are} given in Theorem \ref{add0}.
	\end{corollary}
	\section{Numerical examples}\label{numex}
	All the numerical experiments will be done in MATLAB R2018a with the configuration: Intel(R) Core(TM) i9-11900K CPU 3.50GHz and 64.00G RAM.
	For all numerical examples, we use `CPU(s)', `IT' and `RES' {to} denote the running time, the iterative step and the norm of a residual vector (i.e., $\textrm{RES}=||\mathcal {A}{\bf{z}}^{l-1}-{\bf{b}}||$) of the tested algorithms, respectively. We set the maximum iterative number as $1000$, $\kappa_{\boldsymbol{\alpha}}=1000$ and the stopping criteria as $||\mathcal {A}{\bf{z}}^{l-1}-{\bf{b}}|| < \varepsilon$ with $\varepsilon=10^{-11}$.

 Let $\mathcal{A}=\mathcal{E}-\mathcal{F}$, $\mathcal{D}=D{\mathcal{I}_{l}}$ and $\mathcal{L}=L{\mathcal{I}_{l}}$, where $D$ and $-L$ are the diagonal part and the strictly lower triangle part of $M(\mathcal{A})$, respectively. We give several illustrations {with different algorithms listed in} Table \ref{tab1}, where `-' means that it is unavailable.
	
	\begin{table}[htbp]
		\centering
		\caption{Abbreviations and splitting of the tested methods}
		\scalebox{0.6}{
			\begin{tabular}{lll}
				\hline
				Methods & $\mathcal{E}$ & Abbreviations \\
				\hline
				Jacobian-Anderson type method & $\mathcal{D}$ & Jacobian-RAA({$m$}) \\
				GS-Anderson type method & $\mathcal{D}-\mathcal{L}$ & GS-RAA({$m$}) \\
				SOR-Anderson type method & $\frac{1}{\omega}(\mathcal{D}-\omega\mathcal{L}$) & SOR-RAA({$m$}) \\
                Preconditioned method in \cite{cui}&     $\mathcal{D}-\mathcal{L}$    & P$_{max}$  \\
                Preconditioned SOR-type method in \cite{liuc}&     $\frac{1}{\omega}(\mathcal{D}-\omega\mathcal{L}$)    & P$_{\beta}$SOR \\
                Preconditioned SOR-type method in \cite{com}&     $\frac{1}{\omega}(\mathcal{D}-\omega\mathcal{L}$)    & P$_{\alpha}$SOR \\
                Preconditioned AOR-type method in \cite{Cy}&     $\frac{1}{\omega}(\mathcal{D}-r\mathcal{L}$)    & PAOR \\
                Accelerated modified dynamical system method in \cite{wang5}&     -    & AMDS-POWER \\
                Newton's method in \cite{wei2} &     -    & Newton \\
				\hline
		\end{tabular}}%
		\label{tab1}%
	\end{table}%
	\begin{remark}
		If $m=0$, it is noted that the numerical results are given by the TSM.
	\end{remark}
	

We search the optimal parameter $\theta$ from $0.1$ to $1$ in the interval of $0.1$, the parameter $\omega$ from $1$ to $2$ in the interval of $0.1$ for all the SOR-type methods, the {parameters} $\alpha$ and $\beta$ from $0.1$ to $2$ in the interval of $0.1$, the parameter $\tilde{\gamma}$ from $0.1$ to $3$ in the interval of $0.1$, the parameter $\tilde{\beta}$ from $-0.4$ to $0$ in the interval of $0.1$ and the parameter $r$ from $0$ to $1$ in the interval of $0.1$.

	\begin{example}\label{ex1}(\cite{wei2})
		Taking order-$3$  dimension $n$ nonsingular $\mathcal {M}$-tensors $\mathcal {A}=s\mathcal {I}-\mathcal {B}$, where $\mathcal {B}$ is generated randomly by MATLAB, and $s = (1 + \varepsilon )\mathop {\max }\limits_{i = 1,2,...,n} {(\mathcal {B}{{\bf{e}}^2})_i},~\varepsilon  > 0$, and $\varepsilon  =1$.
	\end{example}
	%
	
	\begin{example}\label{ex4}
		Let $\mathcal {A} \in \mathbb{R}^{[3,n]}$. Then entries $a_{111}=a_{nnn}=8.$ For $i=2,3,\ldots,n-1$,  $a_{iii}=8$, $a_{i+1ii}=a_{ii-1i}=a_{iii+1}=-1/3$.
	\end{example}

\begin{example}\label{ex5} (\cite{xie})
		Let $s=n^2$ and $\mathcal{A}=s\mathcal{I}-\mathcal{B}\in \mathbb{R}^{[3,n]}$ with $b_{ijk}=|\sin (i+j+k)|$.
	\end{example}

In numerical experiments, let $\theta=\theta_k$ for all $k$. We set the right-hand side ${\bf{b}}={\bf{e}}$, where  ${\bf{e}}=({1,1,...,1})^{T}$. Furthermore, we take the initial value ${\bf{z}}_0={\bf{e}}$ and ${\bf{z}}_0=\frac{1}{n}{\bf{e}}$ for Examples \ref{ex1}-\ref{ex4} and \ref{ex5}, respectively.

	\begin{table}[htbp]
		\centering
		\caption{Numerical comparison of the proposed algorithms with the variation of $m$.}
		\scalebox{0.6}{
			\begin{tabular}{rcccccccccc}
				\hline
				&       &       & Jacobian-RAA($m$) &       &       & GS-RAA($m$) &       &       & SOR-RAA($m$) &  \\
				\hline
				& $m$     & CPU(s)   & IT    & RES   & CPU(s)   & IT    & RES   & CPU(s)   & IT    & RES \\
				\hline
				Example \ref{ex1}
                & 0     & 0.1442971 & 39    & 9.38E-12 & 0.1424690 & 39    & 8.93E-12 & 0.0929389 & 25    & 5.10E-12 \\
                & 1     & 0.0357574 & 9     & 6.57E-16 & 0.0350616 & 9     & 6.35E-16 & 0.0340594 & 9     & 6.35E-16 \\
                & 2     & \textbf{0.0317609} & 8     & 1.63E-12 & \textbf{0.0300341} & 8     & 1.99E-12 & \textbf{0.0289184} & 8     & 1.99E-12 \\
                & 3     & 0.0441601 & 11    & 6.16E-13 & 0.0389104 & 10    & 1.91E-12 & 0.0385751 & 10    & 1.91E-12 \\
                & 4     & 0.0403306 & 11    & 6.38E-12 & 0.0396004 & 11    & 2.67E-12 & 0.0396004 & 11    & 7.11E-14 \\
				\hline

				Example \ref{ex4}
                & 0     & 0.0552742 & 14    & 2.36E-12 & 0.0449658 & 12    & 1.95E-12 & 0.0445933 & 12    & 1.42E-12 \\
                & 1     & 0.0378777 & 10    & 7.77E-12 & 0.0332321 & 9     & 4.01E-12 & 0.0341879 & 9     & 2.69E-12 \\
                & 2     & 0.0374533 & 10    & 9.40E-12 & 0.0298825 & 8     & 6.82E-12 & 0.0297657 & 8     & 8.11E-12 \\
                & 3     & \textbf{0.0373813} & 10    & 2.07E-12 & \textbf{0.0296058} & 8     & 7.45E-12 & \textbf{0.0291958} & 8     & 6.62E-12 \\
                & 4     & 0.0386835 & 10    & 1.94E-12 & 0.0303037 & 8     & 7.76E-12 & 0.0295858 & 8     & 5.54E-12 \\
				\hline

                Example \ref{ex5}
                & 0     & 0.1466787 & 38    & 7.90E-12 & 0.1415589 & 38    & 7.63E-12 & 0.0794291 & 22    & 6.27E-12 \\
                & 1     & 0.0224974 & 6     & 1.78E-16 & 0.0233235 & 6     & 1.75E-16 & 0.0221309 & 6     & 1.75E-16 \\
                & 2     & 0.0228135 & 6     & 7.50E-13 & 0.0232898 & 6     & 3.31E-13 & 0.0228285 & 6     & 3.31E-13 \\
                & 3     & \textbf{0.0223159} & 6     & 1.88E-12 & \textbf{0.0223231} & 6     & 1.95E-12 & \textbf{0.0220114} & 6     & 1.95E-12 \\
                & 4     & 0.0236190 & 6     & 5.68E-12 & 0.0229646 & 6     & 4.58E-12 & 0.0225216 & 6     & 4.58E-12 \\
				\hline

		\end{tabular}}%
		\label{data}%
	\end{table}%

	\begin{table}[htbp]
		\centering
		\caption{The optimal parameters of the proposed algorithms}
		\scalebox{0.5}{
			\begin{tabular}{lcccc}
				\hline
				&$m$	& Jacobian-RAA($m$) & GS-RAA($m$) & SOR-RAA($m$)   \\
				\hline
				&	& $\theta$ & $\theta$ & ($\theta,~\omega$)   \\
				\hline
				Example \ref{ex1}&	0    & -   & -   & (-,~1.4) \\
				&	1     & 1     & 1     & (1,~1) \\
				&	2     & 1     & 1     & (1,~1) \\
				&	3     & 1   & 0.9     & (0.9,~1) \\
				&	4     & 0.9     & 0.9   & (1,~1) \\
				\hline
				Example \ref{ex4}	&0     & -   & -   & (-,~1.1) \\
				&	1     & 0.7   & 0.5   & (0.8,~1) \\
				&	2     & 0.5   & 0.7   & (0.9,~1) \\
				&	3     & 0.8     & 0.6   & (0.9 ,~1.1) \\
				&	4     & 0.9   & 0.6   & (0.8,~1.1) \\
				\hline
                Example \ref{ex5}	&0     & -   & -   & (-,~1.5) \\
				&	1     & 1   & 1   & (1,~1) \\
				&	2     & 1   & 1   & (1,~1) \\
				&	3     & 1   & 1   & (1,~1) \\
				&	4     & 1   & 1   & (1,~1) \\
				\hline
		\end{tabular}}%
		\label{para}%
	\end{table}%

It is easy to check that $\mathcal{A}$ in Examples \ref{ex1}-\ref{ex5} is a strong $\mathcal{M}$-tensor and the corresponding multilinear systems \eqref{ax=b} have a unique solution. For $n=200$, we report numerical results in Table \ref{data} and the optimal parameters are listed in Table \ref{para}, where `-' means that it is unavailable. It is seen that if $m=2$ and $m=3$, the SOR-Anderson type method performs best for Examples \ref{ex1} and \ref{ex4}-\ref{ex5}, respectively. The proposed algorithm is valid for a just small $m$ and increasing the value of $m$ does not improve the {algorithm any better}. This is a good numerical behavior because a small value of $m$ means that there is not much computation involved for solving the least square problem \eqref{sq}.


%
\begin{table}
		\centering
		\caption{Numerical comparison with different algorithms.}%
		\scalebox{0.36}{
			\begin{tabular}{llrccrccrccrccrccrcc}
				\hline
				&  &       & $n=50$ &       &      & $n=100$ &       &       & $n=300$ &       &       & $n=500$ &       &       & $n=800$ &       &       & $n=1000$ &  \\
				\hline
				& & CPU   & IT    & RES   & CPU   & IT    & RES   & CPU   & IT    & RES   & CPU   & IT    & RES   & CPU   & IT    & RES   & CPU   & IT    & RES \\
				\hline		
    Example \ref{ex1} & Jacobian-RAA(2) & 0.0014176 & 8     & 1.52E-13 & 0.0036408 & 9     & 1.42E-16 & 0.0927113 & 8     & 1.76E-12 & 0.4460543 & 8     & 5.68E-12 & 1.7014795 & 8     & 2.82E-13 & 3.2124102 & 8     & 8.72E-13 \\
    &  GS-RAA(2) & \textbf{0.0009017} & 8     & 5.82E-12 & 0.0033494 & 8     & 1.70E-13 & 0.0918285 & 8     & 3.73E-12 & 0.4346092 & 8     & 7.95E-12 & 1.6860256 & 8     & 1.19E-12 & 3.2483556 & 8     & 1.23E-12 \\
    & SOR-RAA(2) & 0.0009062 & 8     & 5.82E-12 & \textbf{0.0031615} & 8     & 1.70E-13 & \textbf{0.0918195} & 8     & 3.73E-12 & \textbf{0.4238187} & 8     & 7.95E-12 & \textbf{1.6757426} & 8     & 1.19E-12 & \textbf{3.1609756} & 8     & 1.23E-12 \\
    & P$_{max}$\cite{cui} & 0.0029555 & 37    & 8.44E-12 & 0.0197105 & 39    & 5.05E-12 & 0.5692834 & 40    & 6.05E-12 & 2.5258695 & 40    & 8.50E-12 & 10.4789663 & 41    & 5.76E-12 & 42.345704 & 41    & 6.58E-12 \\
    & Newton\cite{wei2} & 0.0040455 & 9     & 6.09E-15 & 0.0410814 & 10    & 1.52E-15 & 1.4280160 & 11    & 3.53E-12 & 9.6178407 & 12    & 3.39E-14 & 61.0195736 & 13    & 5.00E-17 & 146.3799044 & 13    & 1.14E-14 \\
    & P$_{\beta}$SOR\cite{liuc} & 0.0024676 & 26    & 4.15E-12 & 0.0187555 & 25    & 7.21E-12 & 0.3697548 & 25    & 4.00E-12 & 1.6220897 & 24    & 8.22E-12 & 6.8968100 & 24    & 6.87E-12 & 22.4964405 & 24    & 6.64E-12 \\
    & P$_{\alpha}$SOR\cite{com} & 0.0025883 & 26    & 4.14E-12 & 0.0197529 & 25    & 7.22E-12 & 0.3999717 & 25    & 4.00E-12 & 1.7383071 & 24    & 8.22E-12 & 7.3920309 & 24    & 6.87E-12 & 21.7276480 & 24    & 6.64E-12 \\
    & PAOR\cite{Cy} & 0.0022599 & 18    & 1.65E-12 & 0.0103910 & 19    & 8.71E-13 & 0.3220167 & 19    & 5.94E-12 & 1.4678421 & 20    & 1.57E-12 & 6.1186158 & 20    & 2.82E-12 & 22.0594513 & 20    & 3.49E-12 \\
    & AMDS-POWER\cite{wang5} & 0.0271680 & 11    & 4.87E-13 & 0.0825517 & 12    & 1.99E-13 & 1.3996524 & 13    & 8.39E-13 & 9.2789842 & 14    & 1.03E-13 & 55.7186370 & 14    & 1.78E-12 & 137.5516590 & 14    & 6.67E-12 \\
				\hline
   Example \ref{ex4} & Jacobian-RAA(3) & 0.0016846 & 10    & 2.08E-12 & 0.0045141 & 10    & 2.10E-12 & 0.1172716 & 10    & 2.09E-12 & 0.5235989 & 10    & 2.13E-12 & 1.9884066 & 10    & 2.17E-12 & 3.9926114 & 10    & 2.18E-12 \\
   & GS-RAA(3) & 0.0007408 & 8     & 8.49E-12 & 0.0038454 & 8     & 7.82E-12 & 0.0940870 & 8     & 7.35E-12 & \textbf{0.4061717} & 8     & 7.29E-12 & 1.5978363 & 8     & 7.27E-12 & 3.1481464 & 8     & 7.28E-12 \\
   & SOR-RAA(3) & \textbf{0.0006845} & 8     & 6.31E-12 & \textbf{0.0031653} & 8     & 6.53E-12 & \textbf{0.0933679} & 8     & 6.62E-12 & 0.4293982 & 8     & 6.60E-12 & \textbf{1.5916488} & 8     & 6.54E-12 & \textbf{3.1476961} & 8     & 6.51E-12 \\
   & P$_{max}$\cite{cui} & 0.0007588 & 12    & 9.23E-13 & 0.0039853 & 12    & 1.35E-12 & 0.1420790 & 12    & 2.40E-12 & 0.6163451 & 12    & 3.11E-12 & 2.5129816 & 12    & 3.95E-12 & 4.7925411 & 12    & 4.42E-12 \\
   & Newton\cite{wei2} & 0.0017890 & 6     & 2.33E-16 & 0.0233479 & 6     & 3.76E-16 & 0.7671930 & 6     & 6.99E-16 & 4.7142129 & 6     & 9.14E-16 & 26.4334981 & 6     & 1.16E-15 & 62.7697815 & 6     & 1.30E-15 \\
   & P$_{\beta}$SOR\cite{liuc} & 0.0006917 & 12    & 1.42E-12 & 0.0042315 & 12    & 1.42E-12 & 0.1341507 & 12    & 1.42E-12 & 0.6382303 & 12    & 1.42E-12 & 2.4254230 & 12    & 1.42E-12 & 4.8201556 & 12    & 1.42E-12 \\
   & P$_{\alpha}$SOR\cite{com} & 0.0006921 & 12    & 1.42E-12 & 0.0031779 & 12    & 1.42E-12 & 0.1348621 & 12    & 1.42E-12 & 0.6265198 & 12    & 1.42E-12 & 2.4286392 & 12    & 1.42E-12 & 4.6955130 & 12    & 1.42E-12 \\
    & PAOR\cite{Cy}  & 0.0007791 & 12    & 4.55E-12 & 0.0038869 & 12    & 4.70E-12 & 0.1317069 & 12    & 5.26E-12 & 0.6090724 & 12    & 5.77E-12 & 2.4233616 & 12    & 6.45E-12 & 4.8096263 & 12    & 6.87E-12 \\
    & AMDS-POWER\cite{wang5} & 0.0165788 & 10    & 4.42E-12 & 0.0564043 & 10    & 6.49E-12 & 1.1199035 & 11    & 5.07E-13 & 7.0541374 & 11    & 6.58E-13 & 41.1650175 & 11    & 8.34E-13 & 99.2613736 & 11    & 9.33E-13 \\
				\hline
   Example \ref{ex5} & Jacobian-RAA(3) & 0.0009405 & 7     & 6.21E-16 & 0.0027595 & 7     & 2.50E-15 & 0.0739537 & 6     & 1.88E-12 & 0.4335684 & 8     & 3.50E-14 & 1.4806536 & 7     & 9.27E-13 & 2.8143426 & 7     & 4.06E-14 \\
   & GS-RAA(3) & \textbf{0.0005526} & 6     & 1.07E-12 & 0.0029902 & 6     & 5.56E-13 & 0.0707300 & 6     & 1.26E-13 & \textbf{0.3015579} & 6     & 2.06E-13 & 1.3097623 & 6     & 9.68E-14 & 2.4746146 & 6     & 3.85E-14 \\
   & SOR-RAA(3) & 0.0005669 & 6     & 1.07E-12 & \textbf{0.0023638} & 6     & 5.56E-13 & \textbf{0.0703939} & 6     & 1.26E-13 & 0.3213598 & 6     & 2.06E-13 & \textbf{1.2090712} & 6     & 9.68E-14 & \textbf{2.4226491} & 6     & 3.85E-14 \\
   & P$_{max}$\cite{cui} & 0.0024514 & 42    & 8.85E-12 & 0.0141331 & 40    & 8.64E-12 & 0.5210884 & 37    & 6.63E-12 & 2.2904447 & 35    & 7.69E-12 & 8.7408552 & 33    & 9.42E-12 & 17.0999350 & 33    & 6.75E-12 \\
   & Newton\cite{wei2} & 0.0012694 & 4     & 2.29E-12 & 0.0159670 & 4     & 8.22E-13 & 0.5512742 & 4     & 1.56E-13 & 3.1687929 & 4     & 7.24E-14 & 18.6036020 & 4     & 3.57E-14 & 44.7079714 & 4     & 2.56E-14 \\
   & P$_{\beta}$SOR\cite{liuc} & 0.0019610 & 24    & 8.48E-12 & 0.0106446 & 23    & 7.80E-12 & 0.3252724 & 21    & 7.72E-12 & 1.4672593 & 20    & 8.02E-12 & 6.0325093 & 19    & 8.79E-12 & 12.0475225 & 19    & 6.31E-12 \\
   & P$_{\alpha}$SOR\cite{com} & 0.0016387 & 24    & 8.41E-12 & 0.0094868 & 23    & 7.79E-12 & 0.3351531 & 21    & 7.72E-12 & 1.5115867 & 20    & 8.02E-12 & 6.1476560 & 19    & 8.79E-12 & 12.3410751 & 19    & 6.31E-12 \\
   & PAOR\cite{Cy} & 0.0014030 & 19    & 8.61E-12 & 0.0096942 & 18    & 8.23E-12 & 0.2804358 & 17    & 4.03E-12 & 1.2505620 & 16    & 5.16E-12 & 4.9279153 & 15    & 7.08E-12 & 10.0301371 & 15    & 5.06E-12 \\
   & AMDS-POWER\cite{wang5} & 0.0067865 & 7     & 7.25E-14 & 0.0362239 & 6     & 4.90E-12 & 0.5476106 & 6     & 9.53E-13 & 3.4356464 & 6     & 4.43E-13 & 20.5875038 & 6     & 2.19E-13 & 49.6387163 & 6     & 1.57E-13 \\
				\hline
        \label{datacompare2}
		\end{tabular}}%
		
\end{table}

\begin{table}[htbp]
		\centering
		\caption{The optimal parameters of different algorithms.}
		\scalebox{0.65}{
			\begin{tabular}{lccccccc}
				\hline
				 & Jacobian-RAA($m$) & GS-RAA($m$) & SOR-RAA($m$) & P$_{\beta}$SOR\cite{liuc} & P$_{\alpha}$SOR\cite{com} & PAOR\cite{Cy}  & AMDS-POWER\cite{wang5} \\
				\hline
				& ($m,~\theta$) & ($m,~\theta$) & ($m,~\theta$,~$\omega$) & ($\beta,~\omega$) & ($\alpha,~\omega$) & ($\tilde{\gamma},~\tilde{\beta},~\omega,~r$) & $\hat{\beta}$ \\
				\hline
    Example \ref{ex1}   & (2,~1) & (2,~1) & (2,~1,~1) & (1.4,~1.4) & (0.7,~1.4) & (1.2,~-0.3,~1.1,~0.4) & 0.3 \\
    Example \ref{ex4}   & (3,~0.8) & (3,~0.6) & (3,~0.9,~1.1) & (0.9,~1.1) & (1.3,~1.1) & (1.2,~-0.1,~1.1,~0.8) & 0.26 \\
    Example \ref{ex5}   & (3,~1) & (3,~1) & (3,~1,1) & (0.9,~1.5) & (0.6,~1.5) & (1.6,~-0.3,~1.3,~0.3) & 0.3 \\
        \hline
		\end{tabular}}%
		\label{comallpara}%
	\end{table}%

\begin{figure}[htbp]
		\centering
		\subfigure[Example \ref{ex1}]{
			\begin{minipage}[t]{0.52\linewidth}
				\centering
				\includegraphics[scale=0.5]{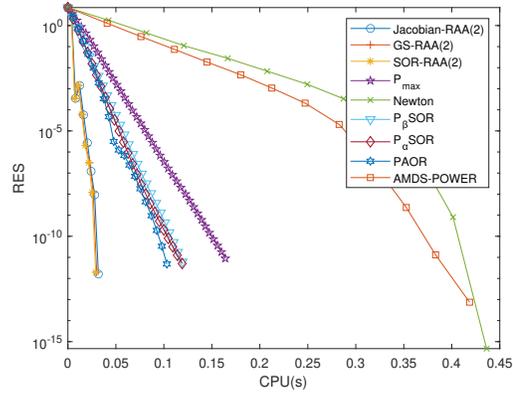}
			\end{minipage}%
		}%
	
		\subfigure[Example \ref{ex4}]{
			\begin{minipage}[t]{0.52\linewidth}
				\centering
				\includegraphics[scale=0.5]{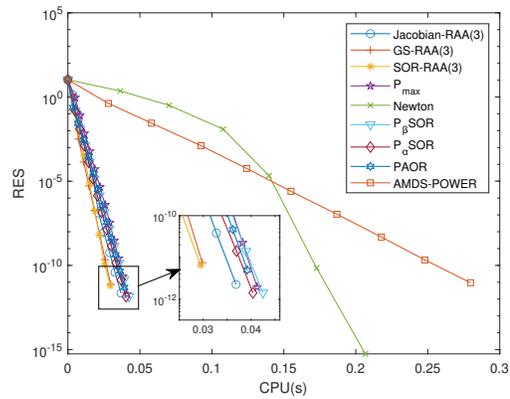}
			\end{minipage}%
		}%
	
        \subfigure[Example \ref{ex5}]{
			\begin{minipage}[t]{0.52\linewidth}
				\centering
				\includegraphics[scale=0.5]{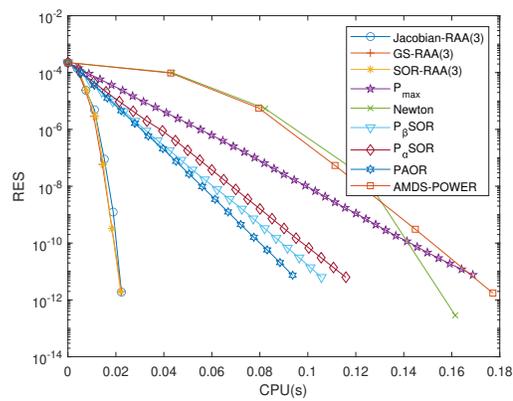}
			\end{minipage}%
		}%


		\centering
		\caption{The relationship between the residual and the running time with different algorithms.} \label{figcom}
	\end{figure}

    For different methods, the numerical results are shown in Table \ref{datacompare2} and the optimal parameters are listed in Table \ref{comallpara}. Besides, Figure \ref{figcom} is given to show the relationship between the norm of a residual vector and the running time. It is seen that the SOR-Anderson type method performs better in running times than other tested algorithms.
	\clearpage
	\section{Conclusion}
	The main contribution {of} this paper is aimed at proposing the restarted nonnegativity preserving tensor splitting methods by the relaxed Anderson acceleration for solving the multilinear systems with a strong $\mathcal{M}$-tensor and the local convergence analysis is also given.  From the numerical experiments, it is seen that the proposed methods with a suitable parameter perform better than the tensor splitting methods without acceleration. Besides, the increasing value of $m$ does not improve the given algorithm better, and the proposed algorithm is valid just for a small value of $m.$  However, how to find an optimal parameter $\theta_k$ is still an open problem, which is one of our future works.
	\\
	\textbf{Declarations}
	\\
	\textbf{Funding}
	The first author was supported in part by National Natural Science Foundation of China (No. 12101136), the Guangdong Basic and Applied Basic Research Foundations (No. 2023A1515011633), the Project of Science and Technology of Guangzhou (No. 2024A04J2056),  the Opening Project of Guangdong Province Key Laboratory of Computational Science at the Sun Yat sen University(No. 2021004), the Open Project of Key Laboratory, School of Mathematical Sciences, Chongqing Normal University (No. CSSXKFKTQ202002). The third author was funded by the Science and Technology Research Program of Chongqing Municipal Education Commission (Grant No. KJQN202100505), the Natural Science Foundation Project of Chongqing (Grant No.
	cstc2021jcyj-msxmX0195), and the Program of Chongqing Innovation Research Group Project in University (Grant No.
	CXQT19018).
	\\
	\textbf{Acknowledgments}
	The authors would like to thank the referees for their helpful comments and suggestions and Prof. Lei Du for sharing their work \cite{du} with us.
	\\
	\textbf{Conflict of interest}
	The authors declare that they have no conflict of interest.

\end{document}